\documentclass[12pt, reqno]{amsart}
\usepackage{amsmath, amsthm, amscd, amsfonts, amssymb, graphicx, xcolor}
\usepackage[bookmarksnumbered, colorlinks, plainpages]{hyperref}

\usepackage{mathtools}
\usepackage{mathrsfs}
\usepackage{bbm} 
\usepackage{enumitem} 

\usepackage[backend=biber,style=numeric]{biblatex} 
\usepackage{graphicx}

\newcommand{\N}{\ensuremath{\mathbb{N}}}

\newcommand{\R}{\ensuremath{\mathbb{R}}}
\newcommand{\C}{\ensuremath{\mathbb{C}}}

\newcommand{\eps}{\ensuremath{\varepsilon}}

\newtheorem{theorem}{Theorem}[section]
\newtheorem{lemma}[theorem]{Lemma}
\newtheorem{proposition}[theorem]{Proposition}
\newtheorem{corollary}[theorem]{Corollary}
\theoremstyle{definition}
\newtheorem{definition}[theorem]{Definition}

\theoremstyle{remark}
\newtheorem{remark}[theorem]{Remark}
\numberwithin{equation}{section}

\begin{document}
\setcounter{page}{1}

\color{darkgray}{
\noindent 

\centerline{}

\centerline{}

\title[Large disks touching three sides of a quadrilateral]{Large disks touching three sides of a quadrilateral}

\author[Alex Rodriguez]{Alex Rodriguez}
\address{Department of Mathematics, Stony Brook University, New York, USA.\\
	\textsc{\newline \indent 
	   \href{https://orcid.org/0000-0001-9097-4025%
	     }{\includegraphics[width=1em,height=1em]{orcid2} {\normalfont https://orcid.org/0000-0001-9097-4025}}
	       }}}
\email{\textcolor[rgb]{0.00,0.00,0.84}{alex.rodriguez@stonybrook.edu}}



\date{December 2, 2024
\newline \indent $^{*}$ Alex Rodriguez. The author is partially supported by NSF grant DMS 2303987 and the Simons Foundation}

\begin{abstract}
We show that every Jordan quadrilateral $Q\subset\C$ contains a disk $D$ so that $\partial D\cap\partial Q$ contains points of three different sides of $Q$. As a consequence, together with some modulus estimates from Lehto and Virtanen, we offer a short proof of the main result obtained by Chrontsios-Garitsis and Hinkkanen in 2024 and it also improves the bounds on their result.
\end{abstract} 

\maketitle

\section{Introduction}

By a Jordan quadrilateral $Q=Q(v_{1},v_{2},v_{3},v_{4})\subset\C$, we mean a bounded Jordan domain with four marked points, which we call \textbf{quad-vertices} (and we suppose that they are oriented counter-clockwise). To abbreviate we will refer to them just as quadrilaterals. We denote the Jordan arcs joining these points by $a_{1}, b_{1}, a_{2}$ and $b_{2}$ as it is indicated in Figure \ref{figure:modulusQuad}, which we call the \textbf{sides} of the quadrilateral. We refer to each pair of non-intersecting sides as opposite sides, and we call $a_{1}, a_{2}$ the $a$-sides of $Q$ and $b_{1},b_{2}$ the $b$-sides of $Q$.

\begin{figure}[h]
	\includegraphics[scale=0.65]{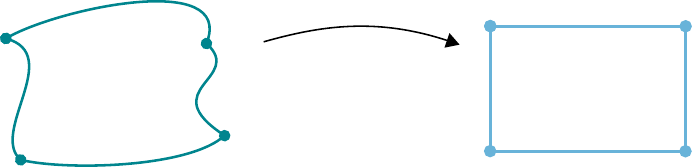}
	\setlength{\unitlength}{\textwidth}
	\put(-0.25,0.115){\scriptsize $\varphi$}
	\put(-0.09,0.055){\scriptsize $R$}
	\put(-0.17,-0.012){\scriptsize $0$}
	\put(0.005,-0.012){\scriptsize $M$}
	\put(-0.17,0.11){\scriptsize $i$}
	\put(0.01,0.11){\scriptsize $M+i$}
	\put(-0.45,0.06){\scriptsize $Q$}
	\put(-0.35,0.05){\scriptsize $b_{1}$}
	\put(-0.55,0.05){\scriptsize $b_{2}$}
	\put(-0.45,0.13){\scriptsize $a_{2}$}
	\put(-0.45,-0.02){\scriptsize $a_{1}$}
	\put(-0.53,-0.01){\scriptsize $v_{1}$}
	\put(-0.35,0.01){\scriptsize $v_{2}$}
	\put(-0.36,0.09){\scriptsize $v_{3}$}
	\put(-0.545,0.095){\scriptsize $v_{4}$}
	\centering
	\caption{Representation of the quad-vertices, $a$-sides and $b$-sides of a quadrilateral Q, together with its representation via conformal mapping to a rectangle.}
	\label{figure:modulusQuad}
\end{figure}
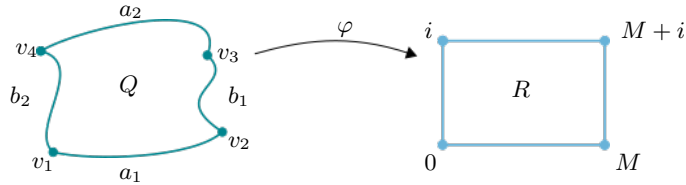

In this paper we prove:

\begin{theorem}\label{mainTheorem}
For any Jordan quadrilateral $\Omega$, there exists a disk $D\subset\Omega$ so that $\partial D\cap\partial Q$ contains points of three sides of $Q$. In particular, it contains points from opposite sides.
\end{theorem}

Here the number three is obviously sharp, in the sense that there are quadrilaterals $Q$ so that $\partial Q\cap\partial D$ does not contain points from more than three sides of $Q$, where $D$ is any disk $D\subset Q$. For instance, in any rectangle that is not a square. If we have a crescent we can also have at most three, with two of them being the endpoints of a pair of adjacent sides, as it can be seen in Figure \ref{figure:crescent}.

\begin{figure}[h]
	\includegraphics[scale=0.6]{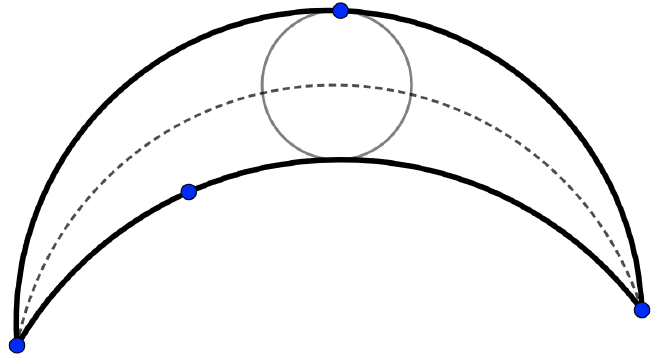}
	\caption{Example of a quadrilateral showing that the number 3 in Theorem \ref{mainTheorem} is sharp, where the figure corresponds to a crescent. The medial axis, which will be defined later in this section, is represented by the dotted curve.}
	\label{figure:crescent}
\end{figure}

We define as the internal distances between the $a$-sides of $Q$ as 
\begin{equation}\label{definition:s_a} s_{a}(Q)\coloneqq\inf\{\textrm{length}(\gamma)\colon\gamma\subset Q \textrm{ Jordan arc joining opposide } a\textrm{-sides}\},\end{equation}
and we similarly define $s_{b}(Q)$.

A quadrilateral can be conformally mapped to a rectangle so that the quad-vertices of the quadrilateral are mapped to the vertices of the rectangle (see 6.2.3 in \cite{AhlforsCA} or Chapter I, Section 2.4 in \cite{LehtoVirtanen}). The ratio between the length of the $a$-side and the $b$-side of this rectangle is a conformal invariant. We call the quantity $M=M(Q)=M(Q(v_{1},v_{2},v_{3},v_{4}))$ given as in Figure \ref{figure:modulusQuad} the \textbf{modulus} of $Q$.

The notion of quasiconformal mapping was introduced by H. Grötzsch in 1928. Since the modulus is a conformal invariant, there is no conformal map that maps a square to a rectangle (which is not a square) mapping vertices to vertices. Grötzsch wanted to find the most nearly conformal mapping that satisfied that. There are several (equivalent) definitions of planar quasiconformal mappings (see for example \cite{AhlforsQC, LehtoVirtanen}). We say that an orientation preserving homeomorphism $\phi\colon\Omega\subset\C\to\Omega^{'}\subset\C$ is $K$-quasiconformal if for every quadrilateral $Q\subset\Omega$, we have $$M(\phi(Q))\leq K M(Q).$$ 

We use Theorem \ref{mainTheorem} to give an alternative proof of the following result of Chrontsios-Garitsis and Hinkkanen in Section \ref{section:applications}.

\begin{corollary}[Theorem 1.1 in \cite{MR4735579}]\label{theorem:Hinkkanen}
For every $K\geq1$ there is a constant $\delta\in(0,1)$ depending only on $K$ such that every Jordan quadrilateral $Q$ with $M(Q)\in[1/K, K]$ contains a disk of radius $\delta\max\{s_{a}(Q),s_{b}(Q)\}$.
\end{corollary}

The notion of modulus of a quadrilateral is closely related to what is commonly known as \textit{modulus of a path family}. By a path family $\Gamma$ we mean a family of curves in $\C$, where each $\gamma\in\Gamma$ is locally rectifiable. We define the modulus of this path family as 
\begin{equation}\label{definition:pathModulus} 
	M(\Gamma)\coloneqq\inf_{\rho \textrm{ admissible}}\int_{Q}\rho^{2}dm(z),
\end{equation}
where we say a non-negative Borel function $\rho$ is admissible if 
$$ L(\rho)\coloneqq\inf_{\gamma\in\Gamma}\int_{\gamma}\rho|dz|\geq1.$$
The modulus is a conformal invariant. Moreover, if we define the path family $\Gamma$ to consist of all those locally rectifiable paths joining opposite $a$-sides of a quadrilateral $Q$, then the modulus of this path family agrees with the modulus of the quadrilateral, i.e. $M(\Gamma)=M(Q)$.

Our approach to Theorem \ref{mainTheorem} is based on using properties of the medial axis. Given any open proper subset $\Omega$ of $\R^{n}$, we define the \textbf{medial axis} as $$ MA(\Omega)=\{z\in\Omega\colon\exists\textrm{ distinct } w, w'\in\partial\Omega \textrm{ s.t. } d(z,w)=d(z,w')=d(z,\partial\Omega)\},$$ i.e. it consists of the centers of balls contained in $\Omega$ so that they intersect $\partial\Omega$ in more than one point. 

Theorem \ref{mainTheorem} proves that the medial axis of a Jordan domain $\Omega\subset\C$ is always non-empty, since for any choice of four distinct points in $\partial\Omega$ there is a disk $D$ so that $\partial D\cap\partial Q$ contains points of three different sides of $Q$, and in particular of two disjoint sides of $Q$, i.e. it contains more than one point. In a much greater generality than what we cover in this document, i.e. for bounded open subsets $\Omega\subset\R^{n}$, in \cite{MR1434446} Fremlin shows that the medial axis is an $F_{\sigma}$ set of Hausdorff dimension at most $n-1$ and that $MA(\Omega)$ is connected if and only if $\Omega$ is. For a simply connected planar domain $\Omega\subset\C$, he shows show that, in addition, it is a union of countably many rectifiable arcs. In this case, in \cite{MR13776} Erd\"os previously proved the bound on the Hausdorff dimension mentioned before.  The medial axis is a subset of the \textbf{central set} of $\Omega$, which is the set of the centers of the maximal balls contained in $\Omega$. In \cite{MR2390513} Bishop and Habokyan prove that this set can have Hausdorff dimension arbitrarily close to $2$, even though the medial axis can have Hausdorff dimension at most $1$.

Choi, Choi and Moon characterized in \cite{MR1491036} the medial axis of a Jordan domain $\Omega\subset\C$ so that $\partial\Omega$ is a finite union of analytic arcs. In the case where $\partial\Omega$ is a finite union of line segments, the medial axis consists of a finite union of analytic arcs each of which is either a line segment or a parabola. 

The medial axis has been proved to be useful in the literature. We cite some references in which this set has been proved to be useful in complex analysis; in \cite{MR2671015} Bishop uses the medial axis to compute the conformal map in linear time, and in \cite{MR2904137} he uses it to find a tree-like decomposition of any simply-connected domain into Lipschitz domains, improving a previously known result by Jones \cite{MR1069238}.

The proof of Theorem \ref{mainTheorem} is structured as follows. In Section \ref{section:preliminaries} we show in Lemma \ref{lemma:limitDisk} that if the result is true for a sequence of quadrilaterals converging from inside to a quadrilateral $Q$, then it is also true for $Q$. We also prove in Lemma \ref{lemma:approxQuad} that any quadrilateral $Q$ is a limit from inside of a sequence of quadrilaterals $\{Q_{n}\}$ so that each $\partial Q_{n}$ is a finite union of line segments that make an angle of $\pi/2$ radians at each one of the quad-vertices.

In \cite{MR1491036,MR1434446} it is proved that the medial axis of a bounded quadrilateral is connected. In Section \ref{section:proofMain} we use this fact to prove Theorem \ref{mainTheorem}, together with the results from Section \ref{section:preliminaries}.

In Section \ref{section:applications}, we prove Corollary \ref{theorem:Hinkkanen} as an application of Theorem \ref{mainTheorem}.

\subsection*{Acknowledgments}

I would like to thank Chris Bishop for his continued guidance and support, and for reading early drafts of the paper which improved the presentation via his comments and corrections. I would also like to thank Dimitrios Ntalampekos for his comments and corrections, which also improved the presentation. The typos and suggestions that the reviewer has highlighted also improved the paper.

\section{Preliminaries}\label{section:preliminaries}

Given any quadrilateral $Q$ we will approximate it first with a sequence of quadrilaterals $Q_{n}$ from inside, so that they satisfy Theorem \ref{mainTheorem}. We then prove that this convergence allows us to conclude that the original quadrilateral also has such a disk so that the boundary intersects the boundary of $Q$ on three sides. 

\begin{definition}[Convergence from inside]
We say that a sequence of quadrilaterals $\{Q_{n}\}$ with sides $a_{j}^{n}, b_{j}^{n}$ for $j=1,2$ converge to the quadrilateral $Q$ from inside if
\begin{enumerate}[label=(\roman*)]
	\item $\overline{Q_{n}}\subset\overline{Q}$ for all $n$.
	\item For all $\eps>0$, there exists $n_{\eps}$ so that for $n\geq n_{\eps}$ and $j\in\{1,2\}$, the Hausdorff distance (see p. 280-281 in \cite{Munkres}) between $a_{j}^{n}$ (resp. $b_{j}^{n}$) and $a_{j}$ (resp. $b_{j}$) is less than $\eps$.
\end{enumerate}
\end{definition}

\begin{figure}[h]
	\includegraphics[scale=1]{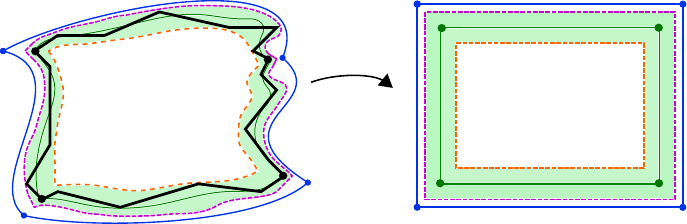}
	\setlength{\unitlength}{\textwidth}
	\centering
	\caption{Approximation of any quadrilateral $Q$ from inside via polygons as in Lemma \ref{lemma:approxQuad}. The dashed curves represent the quadrilaterals converging from inside to $Q$, whereas the thick curve represents the polygon obtained in Lemma \ref{lemma:approxQuad}.}
	\label{figure:approxQuad}
\end{figure}

Convergence from inside implies convergence of the corresponding moduli, as we see in Lemma \ref{lemma:LehtoVirtanen}. The proof is the same as in \cite{LehtoVirtanen}, which we include in this document for completeness.

\begin{lemma}[Lemma 4.3, p.26 in \cite{LehtoVirtanen}]\label{lemma:LehtoVirtanen}
If a sequence of quadrilaterals $Q_{n}$ converges from inside to the quadrilateral $Q$, then $M(Q_{n})\to M(Q)$.
\end{lemma}
\begin{proof}
Take $\varphi\colon Q\to R$ conformal as in Figure \ref{figure:modulusQuad}, then $\varphi$ is uniformly continuous in $Q$, thus the sequence of quadrilaterals $\varphi(Q_{n})$ converges from inside to the quadrilateral $\varphi(Q)=R$. This means that for $\eps>0$ small enough we have $s_{a}(\varphi(Q_{n}))\geq 1-2\eps$ and $s_{b}(\varphi(Q_{n}))\geq M(Q)-2\eps$. Also, $m(\varphi(Q_{n}))\leq m(\varphi(Q))=M(Q)$. Therefore, if in (\ref{definition:pathModulus}) we take the Euclidean metric normalized so that it is admissible, i.e. $\rho(z)=C|z|$, we have
$$\frac{\left(M(Q)-2\eps\right)^{2}}{M(Q)} \leq M(\varphi(Q_{n}))\leq\frac{M(Q)}{\left(1-2\eps\right)^{2}},$$
where we have used that $M(Q(v_{2},v_{3},v_{4},v_{1}))=1/M(Q)=1/M(Q(v_{1},v_{2},v_{3},v_{4}))$ (that is, the modulus of the conjugate quadrilateral is the reciprocal of the original modulus). Since $M(\varphi(Q_{n}))=M(Q_{n})$, we see that $M(Q_{n})\to M(Q)$ as $\eps\to0$.
\end{proof}

Lemma \ref{lemma:LehtoVirtanen} also provides a way to regard any quadrilateral $Q$ as a limit from inside of, for example, analytic quadrilaterals or polygonal quadrilaterals. More precisely, we have the following result. 

\begin{lemma}\label{lemma:approxQuad}
Any quadrilateral $Q$ is a limit from inside of a nested sequence of quadrilaterals $Q_{n}$ so that the boundary of each one of the $Q_{n}$ is a finite union of non-trivial line segments. Moreover, $Q_{n}$ can be taken so that the two segments that meet at the quad-vertices of each one of the quadrilaterals $Q_{n}$ form an angle of $\pi/2$ radians.
\end{lemma}
\begin{proof}
Let $\varphi\colon Q\to R$ be conformal, where $R$ is a rectangle with vertices $0,M,i$ and $M+i$, so that the quad-vertices of the quadrilateral are mapped to the vertices of the rectangle. Define $R_{n}=\{z\in R\colon d(z,\partial R)\geq 1/n\}$ for $n\in\N$. Then $R_{n}$ is a compact subset of $R$ and the sequence of quadrilaterals $\{R_{n}\}_{n\in\N}$ converges to the quadrilateral $R$ from inside. Where we take the quad-vertex $v_{j}$ of the quadrilateral $R_{n}$, as the vertex of the rectangle $R_{n}$ that is the closest to the image under $\varphi$ of the quad-vertex $v_{j}$ of the quadrilateral $Q$ (for $j=1,2,3,4$). Therefore, the sequence $\{\varphi^{-1}(R_{n})\}$ converges to $Q$ from inside.

We now obtain the sequence of polygons that approximates $Q$ from the inside. By Lemma 2 in \cite{MR565480} we can approximate $\varphi^{-1}(\partial R_{2n})$ via a polygonal Jordan curve $\tilde{P}_{n}$ that is contained in $\varphi^{-1}\left(R_{2n+1}\setminus \overline{R_{2n-1}}\right)$. Given this approximation of $\varphi^{-1}(\partial R_{2n})$ by a polygonal $\tilde{P}_{n}$, we can now modify it within the annulus $\varphi^{-1}\left(R_{2n+1}\setminus \overline{R_{2n-1}}\right)$ so that it contains the quad-vertices of the quadrilateral $\varphi^{-1}(R_{2n})$ and so that at the quad-vertices of the quadrilateral $\varphi^{-1}(R_{2n})$ there are two segments of the polygonal curve that meet at an angle of $\pi/2$ degrees. This yields a polygonal $P_{n}$. We define $Q_{n}$ as the interior of the polygonal Jordan curve $P_{n}$ and we take as quad-vertices the quad-vertices of the quadrilateral $\varphi^{-1}(R_{2n})$, which are part of the polygonal by construction. This procedure has been illustrated in Figure \ref{figure:approxQuad}. Also, if two adjacent line segments of $\partial Q_{n}$ are part of the same line, we consider them as only one line segment.
This new sequence of quadrilaterals $\{Q_{n}\}_{n\in\N}$ converges to $Q$ from inside by construction.
\end{proof}

Lemma 2 in \cite{MR4735579} can also be used to construct the polygonal approximations given in Lemma \ref{lemma:approxQuad}.

Next we will show that if Theorem \ref{mainTheorem} holds for a sequence of quadrilaterals that converge from inside to a quadrilateral $Q$, then the same holds for $Q$.

\begin{lemma}\label{lemma:limitDisk}
Let $Q$ be a quadrilateral and suppose that $Q_{n}$ is a increasing sequence of quadrilaterals converging to $Q$ from inside. If every $Q_{n}$ has a disk $D_{n}$ so that $\partial Q_{n}\cap\partial D_{n}$ contains points of three sides of $Q_{n}$, then the same holds for $Q$.
\end{lemma}
\begin{proof}
Let $D_{n}=D(c_{n},r_{n})$ be the sequence of corresponding disks. By passing to a subsequence and changing the labeling of the sides of $Q$ if necessary, we can suppose that their boundary circles all have points on opposite $a$-sides and on the side $b_{1}$. Observe that:
\begin{enumerate}[label=(\roman*)]
	\item $\textrm{diam}(Q)\geq \textrm{diam}(Q_{n})\geq 2 r_{n}\geq s_{a}(Q_{n})$, and $s_{a}(Q_{n})\to s_{a}(Q)>0$. So $\{r_{n}\}$ is a bounded sequence, bounded below away from $0$. 
	\item Given $\eps>0$ small enough, there exists $N$ so that for $n\geq N$ we have $c_{n}\in\{z\in Q\colon \textrm{dist}(z,\partial Q)\geq (s_{a}(Q)-\eps)/2\}=K_{\eps}$, which is a compact set. Observe that if $\eps'<\eps$, then $K_{\eps}\supset K_{\eps'}$.
	\item For each $D_{n}=D(c_{n},r_{n})$ we have at least points $p_{n}\in a_{1}^{n}$, $q_{n}\in a_{2}^{n}$ and $w_{n}\in b_{1}^{n}$ on $\partial D_{n}$, where $a_{1}^{n}$ is the $a_{1}$ side of $Q_{n}$, $a_{2}^{n}$ is the $a_{2}$ side of $Q_{n}$ and $b_{1}^{n}$ is the $b_{1}$ side of $Q_{n}$. 
\end{enumerate}
Take a subsequence so that $c_{n_{j}}, r_{n_{j}}, p_{n_{j}}, q_{n_{j}}, w_{n_{j}}$ converge to $c\in K_{\eps}, r>0, p\in a_{1}, q \in a_{2}$ and $w\in b_{1}$ (now referring to the corresponding sides of $Q$). Then the disk $D=D(c, r)$ is contained in $Q$ and $\{p, q, w\}\subset \partial D\cap\partial Q$, which are points of three sides of $Q$.
\end{proof}

Observe that in the previous Lemma we are not excluding the possibility of $w$ being equal to either $p$ or $q$. This is the case, for example, if we have a crescent, as it can be seen in Figure \ref{figure:crescent}.

\section{Proof of the main theorem}\label{section:proofMain}

By Lemma \ref{lemma:approxQuad}, given any quadrilateral we can build a sequence of quadrilaterals $Q_{n}$, so that $\partial Q_{n}$ is a finite union of non-trivial line segments, with the two segments meeting at each one of the quad-vertices making an angle of $\pi/2$ radians. By Lemma \ref{lemma:limitDisk}, if we show that Theorem \ref{mainTheorem} holds for such quadrilaterals, then the Theorem is proved. This is what we prove in Lemma \ref{lemma:main}. Before we state and prove Lemma \ref{lemma:main}, we need some facts from the medial axis of a polygonal quadrilateral $Q$.

Given any two $c, \tilde{c}\in MA(Q)$, where $MA(Q)\subset\overline{Q}$ is the medial axis of a polygonal Jordan domain $Q$, by Theorem 7.3 in \cite{MR1491036} (or Theorem 2D in \cite{MR1434446}) the medial axis is path-connected. This means that there is a (continuous) path $\gamma\colon[0,1]\to MA(Q)\subset\overline{Q}$ so that $\gamma(0)=c$ and $\gamma(1)=\tilde{c}$. Moreover, by Corollary 8.1 in \cite{MR1491036}, $MA(Q)$ is a strong deformation retract of $Q$, i.e. $MA(Q)$ has a tree-like structure. As it was mentioned before, it is enough to prove Theorem \ref{mainTheorem} for the particular case in which $Q$ is a polygonal quadrilateral so that the two line segments meeting at each one of the quad-vertices make an angle of $\pi/2$ radians. Since $\pi/2<\pi$, then every quad-vertex of the quadrilateral $Q$ is a limit point of $MA(Q)$. Therefore, there is a simple path $\gamma\colon[0,1]\to MA(Q)\subset\overline{Q}$ so that $\gamma(0)=v_{1}$ and $\gamma(1)=v_{3}$ (opposite quad-vertices of the quadrilateral). This simple path $\gamma$ is represented in Figure \ref{figure:exMedialAxis}. Since the distance to the boundary function is continuous, given our choice of $\gamma$, the function $t\mapsto r_{t}$, where $r_{t}$ is the maximal radius so that $D(\gamma(t),r_{t})\subset\overline{Q}$ (i.e. the distance of $\gamma(t)$ to the boundary), is also continuous. Hence if we define a topology on the space of maximal disks by taking the product topology on $\overline{Q}\times[0,\textrm{diam}(Q)]$, then the function $t\mapsto(\gamma(t),r_{t})$ is continuous, i.e. the maximal disks change \textit{continuously}.

\begin{figure}[h]
	\includegraphics[scale=0.9]{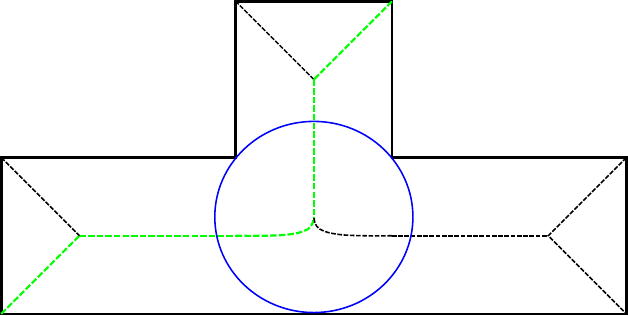}
	\setlength{\unitlength}{\textwidth}
	\put(-0.69,0.0){\scriptsize $v_{1}$}
	\put(-0.237,0.33){\scriptsize $v_{3}$}
	\put(-0.6637,-0.005){\textcolor{green}{$\bullet$}}
	\put(-0.255,0.324){\textcolor{green}{$\bullet$}}
	\put(-0.5,0.1){\scriptsize \textcolor{green}{$\gamma$}}
	\centering
	\caption{Illustration of the proof of Lemma \ref{lemma:main}. The medial axis is the union of the dashed curves. The circle has its center in the medial axis. The simple path $\gamma$ joining $v_{1}$ and $v_{3}$ within the medial axis is represented in a lighter color.}
	\label{figure:exMedialAxis}
\end{figure}
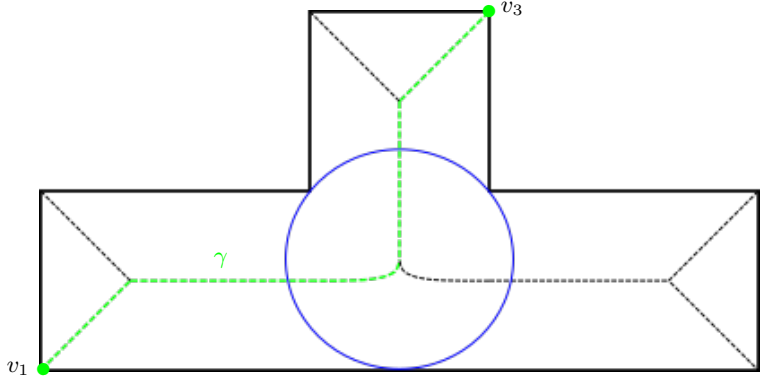

It can also be proved that the simple path $\gamma\colon[0,1]\to MA(Q)\subset\overline{Q}$ so that $\gamma(0)=v_{1}$ and $\gamma(1)=v_{3}$, as we considered before, satisfies the following: given the maximal disk $D_{t}=D(\gamma(t),r_{t})$, then $\partial D_{t}$ intersects the two connected components of $\partial Q\setminus\{v_{1},v_{3}\}$.

We can now prove Theorem \ref{mainTheorem} when $Q$ is a polygonal Jordan domain so that the two line segments that meet at the quad-vertices of $Q$ make an angle of $\pi/2$.

\begin{lemma}\label{lemma:main}
Let $Q=Q(v_{1},v_{2},v_{3},v_{4})$ be a polygonal quadrilateral so that $\partial Q$ is a finite union of non-trivial line segments that form an angle of $\pi/2$ radians at each one of the quad-vertices $v_{j}$. Then there exists a disk $D\subset Q$ so that $\partial D\cap\partial Q$ contains points of three sides of the quadrilateral $Q$.
\end{lemma}
\begin{proof}
Take a simple path $\gamma\colon[0,1]\to MA(Q)\subset\overline{Q}$ so that $\gamma(0)=v_{1}$ and $\gamma(0)=v_{3}$, opposite quad-vertices of $Q$. Close to $v_{1}$, the maximal disks around those points intersect the two line segments that are adjacent to $v_{1}$ (by the $\pi/2$ angle), and in particular, to the two adjacent sides of the quadrilateral to $v_{1}$. Close to $v_{3}$, by the $\pi/2$ angle, the maximal disks around those points intersect the two line segments that are adjacent to $v_{3}$. Since $\gamma$ is path joining those points, there is one last $s\in(0,1)$ for which the corresponding maximal disks intersect the two sides of the quadrilateral $Q$ that are adjacent to $v_{1}$. For such $s$, the maximal disk $D=D(\gamma(s),r_{s})$ is so that $\partial D\cap \partial Q$ contains points from three sides of $Q$.
\end{proof}

\begin{proof}[Proof of Theorem \ref{mainTheorem}]
Observe that Lemma \ref{lemma:main} is precisely the particular case of Theorem \ref{mainTheorem} for the quadrilaterals obtained in Lemma \ref{lemma:approxQuad}. Therefore, by Lemma \ref{lemma:limitDisk}, Theorem \ref{mainTheorem} is proved.
\end{proof}

\section{Proof of Corollary \ref{theorem:Hinkkanen}}\label{section:applications}

The characterization of the modulus in (\ref{definition:pathModulus}) leads to many interesting applications. It yields bounds on the modulus provided that we know some geometric properties of the quadrilateral by choosing adequate admissible metrics. For example, in \cite{LehtoVirtanen} Lehto and Virtanen prove the following:

\begin{proposition}[Lemma 4.1, p. 23 in \cite{LehtoVirtanen}]\label{Proposition:LehtoVirtanen}
The modulus of a quadrilateral $Q$ satisfies the following double inequality
$$ \frac{\left(\log(1+2s_{b}(Q)/s_{a}(Q))\right)^{2}}{\pi+2\pi\log(1+2s_{b}(Q)/s_{a}(Q))}\leq M(Q)\leq\frac{\pi+2\pi\log(1+2s_{a}(Q)/s_{b}(Q))}{\left(\log(1+2s_{a}(Q)/s_{b}(Q))\right)^{2}},$$
where $s_{a}(Q)$ and $s_{b}(Q)$ are defined as in (\ref{definition:s_a}).
\end{proposition}

Observe that Proposition \ref{Proposition:LehtoVirtanen} shows that having a family of quadrilaterals $Q$ with uniformly bounded modulus $M(Q)\in[1/K,K]$ is equivalent to having a uniformly bounded ratio between the internal distances, that is, $s_{a}(Q)/s_{b}(Q)\in[1/L,L]$, where $L$ depends only on $K$. 

Let's see how Theorem \ref{mainTheorem} implies Corollary \ref{theorem:Hinkkanen}. Let $C$ the boundary of the disk $D$ given in Theorem \ref{mainTheorem}, as represented in Figure \ref{figure:quadANDcircle}. Then, there is a segment $\gamma$, contained in $D$, joining opposite sides. Suppose it joins the $a$-sides. We have $\textrm{diam}(C)\geq\textrm{length}(\gamma)\geq s_{a}(Q)$. Thus there exists a disk with the same center as $D$, and of radius $s_{a}(Q)/2$ contained in $Q$. Since we are considering quadrilaterals with uniformly bounded modulus, then by Proposition  \ref{Proposition:LehtoVirtanen}, $$ s_{a}(Q)\geq s_{b}(Q)/L.$$

\begin{figure}[h]
	\includegraphics[scale=1]{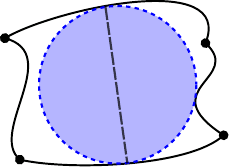}
	\setlength{\unitlength}{\textwidth}
	\put(-0.1,0.1){\scriptsize $Q$}
	\put(-0.146,0.08){\scriptsize $\gamma$}
	\put(-0.1234,-0.002){\scriptsize $\bullet$}
	\put(-0.1495,0.1824){\scriptsize $\bullet$}
	\put(-0.043,0.074){\scriptsize $\bullet$}
	\centering
	\caption{Representation of the circle obtained in Theorem \ref{mainTheorem} that intersects three sides of our quadrilateral $Q$.}
	\label{figure:quadANDcircle}
\end{figure}
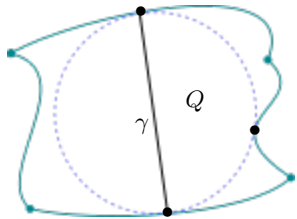

That is, it also contains a disk of radius $s_{b}(Q)/2L$. Therefore, our quadrilateral $Q$ contains a disk of radius $\delta\max\{s_{a}(Q),s_{b}(Q)\}$, where $\delta=1/2L\in(0,1)$ and as we have mentioned before, this constant $L$ only depends on the bound $K$ for the modulus.

This completes the proof of Corollary \ref{theorem:Hinkkanen}.

\begin{remark}
The constant obtained in this new proof of Corollary \ref{theorem:Hinkkanen} is sharp with respect to $L$, i.e. with respect to the quadrilaterals $Q$ so that $s_{a}(Q)/s_{b}(Q)\in [1/L, L]$ (as it can be seen by considering a rectangle). But it is not with respect to $K$; given $K$, Proposition \ref{Proposition:LehtoVirtanen} yields some $L$, but given this $L$, Proposition \ref{Proposition:LehtoVirtanen} gives $\tilde{K}>K$.
\end{remark}

\printbibliography

\end{document}